\documentclass[11pt]{article}
\usepackage[top=2.54cm, bottom=2.54cm, left=2.5cm, right=2.5cm]{geometry}
\usepackage[tbtags]{amsmath}
\usepackage{amssymb}
\usepackage{amsthm}
\usepackage{appendix}
\usepackage{latexsym}
\usepackage{mathrsfs}
\usepackage{xcolor}
\usepackage{wasysym}
\usepackage{fancyhdr}
\usepackage[colorlinks=true, linkcolor=red, citecolor=blue, urlcolor=magenta]{hyperref}
\usepackage{float}
\usepackage{graphicx}

\usepackage[numbers,sort&compress]{natbib}
\usepackage{pict2e,color,xcolor}
\usepackage{tikz}
\usepackage{tikz-3dplot}
\usepackage{tkz-graph}
\usetikzlibrary{arrows,shapes,chains}
\usepackage{enumerate}
\usepackage{verbatim}

\allowdisplaybreaks[4]

\newcommand{\RR}{\mathbb{R}}
\newcommand{\EE}{\mathbb{E}}
\newcommand{\PP}{\mathbb{P}}

\renewcommand{\phi}{\varphi}
\renewcommand{\epsilon}{\varepsilon}
\newcommand{\N}{\mathcal{N}}

\def\bx{\boldsymbol{x}}
\def\by{\boldsymbol{y}}

\newtheorem{theorem}{Theorem}[section]
\newtheorem{lemma}[theorem]{Lemma}
\newtheorem{proposition}[theorem]{Proposition}

\theoremstyle{definition}
\newtheorem{definition}[theorem]{Definition}

\newtheorem{claim}[theorem]{Claim}

\begin{document}

\title{Sharper Ramsey lower bounds from refined Gaussian estimates}
\author{
Qizhong Lin\footnote{Center for Discrete Mathematics, Fuzhou University,
Fuzhou, 350108 P.~R.~China. Email: {\tt linqizhong@fzu.edu.cn}. Supported in part  by National Key R\&D Program of China (Grant No. 2023YFA1010202) and NSFC (No.\ 12571361).} 
\;\; and\;\;
Lin Niu\footnote{Center for Discrete Mathematics, Fuzhou University, Fuzhou, 350108 P.~R.~China. Email: {\tt 1539166573@qq.com}.} 
}
\date{}

\maketitle

\begin{abstract}
Recently, Ma, Shen and Xie broke the Erd\H{o}s barrier for off-diagonal Ramsey numbers $R(\ell,C\ell)$, achieving the first exponential improvement over the classical lower bound for every $C>1$ and sufficiently large $\ell$. Hunter, Milojevi\'{c}, and Sudakov later gave a simplified proof using Gaussian random graphs and obtained better quantitative bounds. In this paper we prove a further improvement, and show that the exponent in the Ramsey lower bound can be increased by a strictly positive amount for every fixed $C>1$; as $C\to\infty$, the gain is asymptotically $\Theta(p_C^{-1/2}/\log C)$. The improvement is achieved by replacing the subgaussian estimate for truncated Gaussians with a sharp cumulant generating function bound.
\end{abstract}

\section{Introduction}
\label{sec:intro}

For positive integers \(k,\ell\ge 1\), the Ramsey number \(R(\ell,k)\) is the smallest integer \(n\) such that every red/blue coloring of the edges of the complete graph \(K_n\) contains either a red copy of \(K_\ell\) or a blue copy of \(K_k\). The existence of \(R(\ell,k)\) follows from Ramsey \cite{ram}. These numbers lie at the heart of Ramsey theory, a central branch of combinatorics with a long and rich history. For further background, see the surveys by Conlon, Fox, and Sudakov \cite{cfs} and by Morris \cite{Morris}.

In 1935, Erd\H{o}s and Szekeres \cite{ES} obtained the first non-trivial upper bound \(R(\ell,k)\le \binom{\ell+k-2}{\ell-1}\), which in the diagonal setting gives \(R(\ell,\ell)\le 4^\ell\). 
For a long time, the problem saw little progress until R\"odl (unpublished) and later Graham and R\"odl~\cite{GR} gave nontrivial improvements. 
Thomason~\cite{T1988} then achieved a polynomial improvement when \(k\) and \(\ell\) are of the same order.
Conlon~\cite{Conlon2009} extended Thomason's quasi-randomness approach to obtain a superpolynomial improvement when \(k\) and \(\ell\) grow together, and Sah \cite{Sah2023} later refined these methods further.
Recently, Campos, Griffiths, Morris and Sahasrabudhe \cite{CGMS} obtained the first exponential improvement over the classical bound, showing that for all integers \(\ell\le k\),
$R(\ell,k)\leq e^{-\ell/400+o(k)} \binom{k+\ell}{\ell}.$
In particular, \(R(\ell,\ell)\le (4-\varepsilon)^\ell\) for some absolute constant \(\varepsilon>0\).
Gupta, Ndiaye, Norin and Wei \cite{GNNW} refined the CGMS argument and improved the bounds further, showing
$R(\ell,k)\leq e^{-\ell/20+o(k)} \binom{k+\ell}{\ell}$
for all \(\ell\le k\); consequently, \(R(\ell,\ell)\le 3.8^{\ell+o(\ell)}\). 
Very recently, Balister, Bollob\'as, Campos, Griffiths, Hurley, Morris, Sahasrabudhe, and Tiba \cite{BBCGHMST} gave an alternative proof, which also yields exponential improvements in multicolor settings.

While upper bounds have seen dramatic progress, lower bounds have witnessed a different kind of prosperity.
A classic result of Erd\H{o}s \cite{E1947} states that for any fixed \(C\ge 1\),
\[
R(\ell,C\ell) = \Omega\left(\ell \cdot p_C^{-\ell/2}\right),\quad 
p_C\in(0,1/2] \; \text{ solves } \; C = \frac{\log p_C}{\log(1-p_C)}.
\]
For \(C=1\) this gives the classical diagonal bound \(R(\ell,\ell)\ge \frac{\ell}{e\sqrt{2}}2^{\ell/2}\). 
Despite Spencer's modest improvement by a factor of two using the Lov\'asz Local Lemma \cite{Spencer1975}, Erd\H{o}s' bound remained essentially unchanged for decades.

The situation is quite different when one of the two parameters is fixed, and substantial progress has been made in this case. For \(\ell = 3\), improving on the lower bound of Erd\H{o}s \cite{E1961}, Kim \cite{Kim} firstly determined the correct order of \(R(3,k)\), matching the upper bound due to Ajtai, Koml\'os, and Szemer\'edi \cite{AKSz}:
$
R(3,k) = \Theta( k^2 / \log k).
$ Subsequent advances \cite{BK10,BK21,PGM,CJMS,HHKP} improved the lower bound to \(R(3,k)\ge (\frac{1}{2}+o(1))k^2/\log k\). In particular, Erd\H{o}s conjectured that for every fixed \(\ell\ge 4\), there exists a constant \(c>0\) such that \(R(\ell,k)> k^{\ell-1}/\log^c k\). For \(\ell=4\), a breakthrough of Mattheus and Verstra\"ete \cite{M-V} showed that \(R(4,k)=\Omega(k^3/\log^4 k)\). For every fixed \(\ell\ge 5\), despite considerable progress \cite{AKSz,BK10,lrz,Spencer1977}, the precise order of magnitude of \(R(\ell,k)\) remained unknown until recently, when Brada\v{c} \cite{bradac} achieved a breakthrough, proving that \(R(\ell,k)> k^{\ell-1}/\log^{2\ell-4} k\) for each \(\ell\ge5\).

While the case where one parameter is fixed has seen substantial progress, the situation where $\ell$ and $k$ grow together (i.e., $k=C\ell$) remained largely open for decades. This barrier was recently broken by a remarkable work of Ma, Shen, and Xie \cite{MSX}, who achieved the first exponential improvement over the classical lower bound. They showed that for every fixed $C>1$ and all sufficiently large $\ell$,
\[
R(\ell,C\ell)\ge \left(p_C^{-1/2}+\epsilon_{\mathrm{MSX}}\right)^\ell
\]
for some constant $\epsilon_{\mathrm{MSX}}>0$ depending on $C>1$.

Very recently, Hunter, Milojevi\'{c}, and Sudakov \cite{HMS} introduced an elegant alternative proof that applies the Gaussian random graph model instead of the random sphere graph model of Ma et al. \cite{MSX}, leading to a much cleaner argument. More importantly, their quantitative bounds improved upon those of MSX; in particular, for large $C$ one may take $\epsilon_{\mathrm{HMS}}=(e^{1/24}-1+o(1))p_C^{-1/2}$. 

In this paper we further refine the HMS analysis. Our key innovation is a sharp cumulant generating function bound for upper-truncated Gaussians (Lemma~\ref{lem:CGF}), which improves upon the subgaussian estimate used in \cite{HMS}. As a result, we obtain a strictly larger exponent \(\eta > \epsilon_{\mathrm{HMS}}\).

\begin{theorem}\label{thm:main}
For every \(C>1\) there is a constant \(\eta>\epsilon_{\mathrm{HMS}}\) such that for all sufficiently large \(\ell\),
\[
R(\ell,C\ell)\ge \left(p_C^{-1/2}+\eta\right)^\ell,
\]
where \(p_C\in(0,1/2)\) is as above. 
Moreover, as \(C\to\infty\) we may take 
\[
\eta = \left(e^{1/24}-1\right)p_C^{-1/2}+\Theta\left({p_C^{-1/2}}\big/{\log C}\right).
\]
\end{theorem}


{\em Remark.}  Brada\v{c} \cite{bradac} proved that \(R(\ell,k)\ge (k/\ell)^{(1-o(1))\ell}\) for \(k\gg \ell\), which improves earlier bounds in the far off-diagonal regime. In addition, for fixed \(C>1\), $R(\ell,C\ell)\ge (2^{1-\frac{1}{2C}})^\ell.$ In the near-diagonal regime \(C\to1^+\), this bound improves upon \cite{MSX,HMS} and Theorem~\ref{thm:main}.


\medskip
The rest of the paper is organized as follows.
Section~\ref{sec:preliminaries} collects preliminary facts about Gaussian distributions, concentration inequalities, truncated Gaussians and subgaussian estimates; it also contains the formal definition of the Gaussian random graph and the basic probability quantities we will use. 
Section~\ref{sec:setup} introduces the notion of perfect sequences and the Bartlett decomposition, which together form the geometric and algebraic framework for the reverse induction at the core of the HMS proof. 
In Section~\ref{sec:cgf} we prove the refined cumulant generating function bound for truncated Gaussians, which is the main novelty of this work. 
Section~\ref{sec:improved} incorporates this improvement into the induction to obtain a sharper independent set probability. 
Finally, Section~\ref{sec:proof} completes the proof of Theorem~\ref{thm:main}.

\medskip
\noindent\textbf{Notation.}
In the remainder of the paper, we write $O(\cdot)$ (or $O_p(\cdot)$) to denote a quantity bounded by a constant that may depend on $p$ (and hence on \(C\), since \(p=p_C\) is determined by \(C\)), but is independent of \(\ell\), \(d\), \(D\). We also write $o_\ell(\cdot)$ to denote a quantity that tends to zero as $\ell\to\infty$ while $D$ remains fixed.

\section{Preliminaries}
\label{sec:preliminaries}

Our proof follows the general framework of Hunter, Milojevi\'{c} and Sudakov \cite{HMS}, which itself simplified and refined the approach of Ma, Shen and Xie \cite{MSX}. 
The key idea is to use the Bartlett decomposition to perform a reverse induction on the columns of the Gram matrix, reducing the problem of estimating clique probabilities to controlling exponential moments of quadratic forms in truncated Gaussians.

Before presenting the main proof, we collect several standard facts about Gaussian random vectors, define the Gaussian random graph model, and introduce the basic probability quantities that will be used throughout the paper.

\medskip
\noindent\textbf{Gaussian random graph.}
We begin with the formal definition of the Gaussian random graph, which is the central object of our study.

\begin{definition}\label{def:gaussian_graph}
Let \(n,d\) be positive integers, and let \(p\in(0,1/2)\). 
Let \(c_p>0\) be the unique real number for which \(\PP[Z\le -c_p]=p\), where \(Z\sim\mathcal N(0,1)\) is the standard one-dimensional Gaussian. 
The vertices of the Gaussian random geometric graph \(G(n,d,p)\) correspond to vectors \(\bx_1,\ldots,\bx_n\) independently sampled from the \(d\)-dimensional normal distribution \(\mathcal N(0,\frac1d I_d)\). 
Vertices \(i,j\in[n]\) are connected by an edge if \(\langle\bx_i,\bx_j\rangle\ge -c_p/\sqrt d\). 
\end{definition}

Given a graph \(G\sim G(n,d,p)\), we define a red/blue edge-coloring of the complete graph \(K_n\) by declaring the edges of \(G\) to be blue and all other edges to be red.

For a given set of \(r\) vertices, we denote the probability that they form a red clique (i.e., an independent set in \(G\)) and a blue clique (i.e., a clique in \(G\)) as follows:
\[
P_{\mathrm{red},r}=\PP[G(r,d,p)\text{ is an independent set}],\quad 
P_{\mathrm{blue},r}=\PP[G(r,d,p)\text{ is a clique}].
\]
In the Erd\H{o}s-R\'enyi random graph, edges are independent, giving 
\[P_{\mathrm{red},\ell}=p^{\binom{\ell}{2}},\quad P_{\mathrm{blue},C\ell}=(1-p)^{\binom{C\ell}{2}}.\] 
In the Gaussian random graph, however, edges exhibit correlations: independent sets become less likely, while cliques become more likely.

\medskip
\noindent\textbf{Concentration inequalities.}
We first recall the following result, which controls the Euclidean norm of a high-dimensional Gaussian vector~\cite[Lemma~2.1]{HMS}. 

\begin{lemma}[Norm concentration \cite{HMS}]\label{lem:norm}
If \(\bx\sim\mathcal N(0,\frac1d I_d)\) and \(\delta\in(0,1)\), then
\[
\PP\bigl[\|\bx\|_2\in(1-\delta,1+\delta)\bigr]\ge 1-2\exp(-\delta^2 d/10).
\]
\end{lemma}

The next lemma bounds the length of the projection of a Gaussian vector onto a fixed low-dimensional subspace~\cite[Lemma~2.3]{HMS}. 
For a subspace \(W\subseteq\RR^d\), we denote by \(\pi_W:\RR^d\to W\) the orthogonal projection onto \(W\).

\begin{lemma}[Projection tail bound \cite{HMS}]\label{lem:proj}
Let \(1\le s\le C\ell\) and let \(W\) be any \(s\)-dimensional subspace of \(\RR^d\). 
If \(\bx\sim\mathcal N(0,\frac1d I_d)\) and \(\alpha=100C\log(10/p)\), then
\[
\PP\Bigl[\|\pi_W(\bx)\|_2\ge \frac{\alpha\sqrt\ell}{\sqrt d}\Bigr]\le \Bigl(\frac{p}{10}\Bigr)^{10C\ell}.
\]
\end{lemma}

\medskip
\noindent\textbf{Truncated Gaussians.}
Since our analysis heavily involves Gaussian random variables conditioned on being below a certain threshold, we recall some useful facts about truncated Gaussians. 

Let \(\phi\) and \(\Phi\) denote the standard normal density and distribution function, respectively. 
It is known~\cite{gor} that for a standard normal \(Z\), the inverse Mills ratio satisfies for every \(t<0\),
\[
|t|\le \frac{\phi(t)}{\Phi(t)}\le |t|+\frac1{|t|}. 
\]
The function \(\Phi\) is log-concave, thereby $\log \Phi(t+\varepsilon)\leq \log \Phi(t)+\varepsilon (\log \Phi(t))'$, which implies that for any $\epsilon$,
\[
\Phi(t+\epsilon)\le \Phi(t)e^{\epsilon\phi(t)/\Phi(t)}. 
\]

The following lemma gives the expectations of one-sided truncated Gaussians~\cite[Lemma~2.4]{HMS}.
\begin{lemma}[Expectations of truncated Gaussians \cite{HMS}]\label{lem:trunc_exp}
Let \(b,\epsilon\in\RR\), and let \(X\sim\mathcal N(0,1/d)\). Then
\[
\EE\left[X\mid X\le \frac{b}{\sqrt d}\right]=-\frac{\phi(b)}{\Phi(b)\sqrt d},\quad
\EE\left[X\mid X\ge \frac{b}{\sqrt d}\right]=\frac{\phi(b)}{(1-\Phi(b))\sqrt d},
\]
and the same formulas hold with \(b+\epsilon\) up to an additive error \(O(\epsilon/\sqrt d)\).
\end{lemma}

\medskip
\noindent\textbf{Subgaussian estimates.}
A standard tool for bounding exponential moments of sums of random variables is the theory of subgaussian distributions. 
Recall that a random variable \(X\) is subgaussian with variance proxy \(\sigma^2\) if 
\[
\EE\left[e^{\lambda(X-\EE X)}\right]\le e^{\sigma^2\lambda^2/2}\quad\text{for all }\lambda\in\RR.
\]
A useful fact, proved in~\cite{Barreto}, is that truncated Gaussians have variance proxy at most that of the original Gaussian.

The following lemma bounds the exponential moment of the square of a centered subgaussian variable~\cite[Lemma~2.5]{HMS}.

\begin{lemma}[Exponential moment of a square \cite{HMS}]\label{lem:subg_square}
Let \(X\) be a subgaussian with variance proxy \(\sigma^2\) and \(\EE [X]=0\). 
For \(0\le\lambda<\frac1{2\sigma^2}\),
$
\EE[e^{\lambda X^2}]\le 1+\frac{4\lambda\sigma^2}{1-2\lambda\sigma^2}.
$
\end{lemma}

The following lemma, due to \cite[Lemma 2.6]{HMS}, bounds the exponential moment of a quadratic form $\sum_{i<j} X_i X_j$ in independent subgaussian variables.

\begin{lemma}[Quadratic exponential moment \cite{HMS}]\label{lem:quadratic}
Let \(X_1,\ldots,X_k\) be independent subgaussian variables with variance proxy \(1/d\), and let \(\lambda\in\RR\) satisfy \(d\ge 4|\lambda|k\). 
If \(S=\sum_{1\le i<j\le k}X_iX_j\), then
$
\EE[e^{\lambda S}]\le \exp\left(\lambda\EE[S]+\frac{\lambda^2k^2}{d}\sum_{j=1}^k(\EE X_j)^2+\frac{4|\lambda|k}{d}\right).
$\end{lemma}

\section{Perfect sequences and Bartlett decomposition}
\label{sec:setup}

The Gaussian vectors defining our random graph are not independent in terms of their inner products; however, typical sequences exhibit a certain regularity that simplifies analysis. 
Following \cite{HMS,MSX}, we formalize this notion as follows.
For convenience, we write \(\pi_{i}\) to denote the orthogonal projection onto \(\operatorname{span}\{\bx_1,\dots,\bx_i\}\) (and \(\pi_0\) for the zero map).

\begin{definition}\label{def:perfect}
A sequence of vectors \(\bx_1,\ldots,\bx_r\) is \textbf{perfect} if for all \(1\le i\le r\),
\[
\|\bx_i\|_2\in(1-\delta,1+\delta),\quad \text{and} \quad
\|\pi_{i-1}(\bx_i)\|_2:=\|\pi_{\mathrm{span}\{\bx_1,\ldots,\bx_{i-1}\}}(\bx_i)\|_2\le\frac{\alpha\sqrt\ell}{\sqrt d},
\]
where \(\alpha=100C\log(10/p)\) and \(\delta=\alpha d^{-1/4}\). 
\end{definition}

The definition ensures that each vector has norm close to one and that its projection onto the span of the previous vectors is small. 
Let \(P_{\mathrm{red},r}^*\) (resp. \(P_{\mathrm{blue},r}^*\)) denote the probability that \(\bx_1,\ldots,\bx_r\) form a red clique (resp. blue clique) and are perfect. 
The following proposition (in a slightly different form), proved in \cite[Proposition 3.1]{HMS}, bounds these perfect-sequence probabilities.

\begin{proposition}[Perfect-sequence probabilities \cite{HMS}]\label{prop:perfect_bound}
For \(d=D^2\ell^2\) with \(D\) sufficiently large, and for every \(1\le r\le C\ell\),
\[
P_{\mathrm{red},r}^*\le p^{\binom{r}{2}}\exp\left(-\frac{a^3}{p^3\sqrt{d}} \binom{r}{3} + K\frac{r^3}{D\sqrt{d}} \right),
\]
\[
P_{\mathrm{blue},r}^*\le (1-p)^{\binom{r}{2}}\exp\left(\frac{a^3}{(1-p)^3\sqrt{d}} \binom{r}{3} + K\frac{r^3}{D\sqrt{d}} \right),
\]
where \(a=\phi(c_p)=e^{-c_p^2/2}/\sqrt{2\pi}\), and $K=K(p)$ depends only on $p$. Moreover, the same bound holds for the unconditional probabilities \(P_{\mathrm{red},r}\) and \(P_{\mathrm{blue},r}\), because the probability of failing to be perfect is exponentially small and can be absorbed into the error term.
\end{proposition}

The proof of Proposition~\ref{prop:perfect_bound} proceeds via a reverse induction over the columns of the Bartlett matrix, which we now describe. 
To analyze the inner products between the Gaussian vectors, it is convenient to rotate them so that they become lower-triangular. 
This is achieved by the Bartlett decomposition, which is essentially the Gram-Schmidt process applied to the sequence.

\begin{definition}[Bartlett vectors]
For \(1\le i\le r\le d\), define \(\by_i\in\RR^d\) as follows:

\smallskip
(1) the first \(i-1\) coordinates are independent \(\N(0,1/d)\);

\smallskip
(2) the \(i\)-th coordinate is \(\sqrt{\frac1d\chi_{d-i+1}^2}\), where $\chi_k^2$ denotes the sum of squares of $k$ one-dimensional standard normal random variables;

\smallskip
(3) all further coordinates are zero.
\end{definition}

The Bartlett decomposition preserves the joint distribution of the inner products; see \cite[Lemma 4.2]{HMS}. More precisely, we have the following representation. For a Bartlett vector \(\by_i\), we denote its \(j\)-th coordinate by \(y_i(j)\).
\begin{lemma}[Bartlett representation \cite{HMS}]\label{lem:bartlett}
Let $d\geq r\geq 1$ be positive integers, let $\by_1, \dots, \by_r$ be sampled as above, and let $\bx_1, \dots, \bx_r\sim \mathcal{N}(0, \frac{1}{d}I_d)$ be independent.
The joint distribution of inner products \(\{\langle\bx_i,\bx_j\rangle\}_{1\le i,j\le r}\) is the same as that of \(\{\langle\by_i,\by_j\rangle\}_{1\le i,j\le r}\).

In particular, if $G'$ is a random graph on the vertex set $[r]$ where the vertices \(i\) and \(j\) are adjacent if $\langle \by_i, \by_j\rangle\geq -\frac{c_p}{\sqrt{d}}$, then $G'$ follows the same distribution as the random graph $G(r, d, p)$.
\end{lemma}

We arrange \(\by_1,\ldots,\by_r\) as rows of a lower-triangular matrix \(M\). 
For \(1\le s\le r\), let \(M[s]\) denote the first \(s\) columns. 
The key observation is that if we fix the first \(s\) columns, then the remaining columns are independent and their distributions are simple to describe.
For \(s< i<j\le r\), define
\[
E_{ij}=\{\langle\by_i,\by_j\rangle\ge -c_p/\sqrt d\},\quad \text{and} \quad
\overline{E}_{ij}\text{ its complement},
\]
\[
C_s=\bigwedge_{s< i<j\le r}E_{ij},\quad \text{and} \quad
I_s=\bigwedge_{s< i<j\le r}\overline{E}_{ij}.
\]
Thus \(C_s\) is the event that the vertices \(\{s+1,\ldots,r\}\) form a blue clique, and \(I_s\) is the event that they form a red clique (an independent set).
Finally, let $B_r$ be the event that $\by_1,\dots,\by_r$ form a perfect sequence. In particular, under $B_r$ we have $\|\pi_{i}(\by_{i+1})\|_2\le \frac{\alpha\sqrt\ell}{\sqrt d}$ for all $1\le i\le r-1$.

A key ingredient in the induction is the following bound on connection probabilities \cite{HMS}, which follows from Lemma~\ref{lem:trunc_exp} and the perfect sequence estimates.

\begin{claim}[Connection probabilities \cite{HMS}]\label{claim:connection}
Given $M[s-1]$, the probability (over the random choice of $M_s$) that vertex $s$ is connected to all of $s+1,\dots,r$ satisfies
\[
\mathbb{P}\bigg[\bigwedge_{i=s+1}^r E_{s,i}\;\bigg|\; M[s-1]\bigg]\le (1-p)^{r-s}\exp\bigg(\frac{a\sqrt{d}}{1-p}\sum_{i=s+1}^r\langle\pi_{s-1}(\by_i),\pi_{s-1}(\by_s)\rangle+O((r-s)\delta)\bigg),
\]
and the probability that $s$ is not connected to any of $s+1,\dots,r$ satisfies
\[
\mathbb{P}\bigg[\bigwedge_{i=s+1}^r \overline{E}_{s,i}\;\bigg|\; M[s-1]\bigg]\le p^{r-s}\exp\bigg(-\frac{a\sqrt{d}}{p}\sum_{i=s+1}^r\langle\pi_{s-1}(\by_i),\pi_{s-1}(\by_s)\rangle+O((r-s)\delta)\bigg).
\]
\end{claim}

With these preparations, we can state the core inductive statement, which is proved by reverse induction on $s$; see \cite[Proposition 4.4]{HMS}.

\begin{proposition}[Bartlett induction \cite{HMS}]\label{prop:induction}
Fix the first \(s\) columns \(M[s]\). Then
\[
\PP[I_s\wedge B_r\mid M[s]]\le p^{\binom{r-s}{2}}\exp\left(\begin{aligned}
&-\frac{a\sqrt d}{p}\sum_{s<i<j\le r}\langle\pi_s(\by_i),\pi_s(\by_j)\rangle \\
&\quad -\frac{a^3}{p^3\sqrt d}\binom{r-s}{3}+O\Bigl(\frac{(r-s)^4}{d}+(r-s)\Bigr)
\end{aligned}\right),
\]
and similarly for \(C_s\wedge B_r\) with \(p\) replaced by \(1-p\) and the sign of the linear term flipped.
\end{proposition}

In the following sections, we will refine the key estimate used in the induction step: the bound on the exponential moment of the quadratic form \(\sum_{s<i<j\le r} y_i(s) y_j(s)\) that appears inside the expectation when moving from $s-1$ to $s$.

\section{Refined cumulant generating function for truncated Gaussians}
\label{sec:cgf}

In this section we prove the key technical improvement over the HMS analysis. 
Recall that in the induction step, after conditioning on the event \(\bigwedge_{i=s+1}^r\overline{E}_{s,i}\) (i.e., that vertex \(s\) is not connected to any later vertex), the variables \(y_i(s)\) become independent upper-truncated Gaussians. 
The HMS proof used a subgaussian bound to control the exponential moment of the quadratic form $\sum_{s<i<j\le r}y_i(s)y_j(s)$.  
Replacing it with a sharp cumulant generating function bound (Lemma~\ref{lem:CGF}), which exploits the fact that conditioning strictly reduces the variance, yields an improved estimate.

We begin by studying the variance of a standard normal truncated from above.

\begin{lemma}[Monotonicity of the truncated variance]
\label{lem:truncated_variance_monotone}
Let \(Z\sim \mathcal{N}(0,1)\), and set
\[
V(b):=\operatorname{Var}(Z\mid Z\le b), \quad \text{where} \quad b\in \mathbb{R}.
\]
Then \(V(b)\) is increasing in \(b\), and \(V(b)<1\) for all \(b\le 0\). 
In particular, for \(b=-c_p\) we set
\[
\gamma(p):=1-V(-c_p)>0.
\]
\end{lemma}

\begin{proof}
Recall that \(\phi(z)=\frac1{\sqrt{2\pi}}e^{-z^2/2}\) and \(\Phi(z)=\int_{-\infty}^z\phi(u)\,du\), and let
$
m(b)=\frac{\phi(b)}{\Phi(b)}
$
be the inverse Mills ratio.
By integration by parts, noting that $\phi'(z) = -z\phi(z)$, we have
\begin{align*}
\mathbb E[Z\mid Z\le b]
=
\frac{1}{\Phi(b)}\int_{-\infty}^b z\phi(z)\,dz
=
-\frac{\phi(b)}{\Phi(b)}
=-m(b).
\end{align*}
Moreover,
$
\int_{-\infty}^b z^2\phi(z)\,dz = \bigl[z \cdot (-\phi(z))\bigr]_{-\infty}^b - \int_{-\infty}^b (-\phi(z))\,dz 
=-b\phi(b) + \Phi(b).
$
Therefore,
\begin{align*}
\mathbb E[Z^2\mid Z\le b]
=
\frac{1}{\Phi(b)}\int_{-\infty}^b z^2\phi(z)\,dz
=
\frac{\Phi(b)-b\phi(b)}{\Phi(b)}
=
1-bm(b).
\end{align*}
Consequently,
$
V(b)=\mathbb E[Z^2\mid Z\le b]-(\mathbb E[Z\mid Z\le b])^2=1-bm(b)-m(b)^2.
$
Also, note that $\Phi'(b)=\phi(b)$ and $\phi'(b) = -b\phi(b)$, we have
\[
m'(b)=\frac{\phi'(b)\Phi(b)-\phi(b)\Phi'(b)}{\Phi^2(b)}=\frac{-b\phi(b)\Phi(b)-\phi^2(b)}{\Phi^2(b)}=-bm(b)-m(b)^2.
\]
Thus we obtain
$
V(b)=1+m'(b).
$
Therefore, 
\[
V'(b)=m''(b)
=-m(b)+(b+2m(b))m(b)(b+m(b)) =m(b)\left((b+m(b))^2-V(b)\right).
\]

It remains to show that the factor in parentheses is nonnegative. 
Let
\[
        W_b:=b-Z
\]
under the conditioning \(Z\le b\).  Then \(W_b\ge0\), and its density is
$f_b(w)=\frac{\phi(b-w)}{\Phi(b)}$, where $w\ge0.$ For \(t\ge0\), we have
\[
        \mathbb P(W_b\ge t)
        =
        \int_t^\infty f_b(w)\,dw
        =
        \frac{\Phi(b-t)}{\Phi(b)}.
\]
Since \(\Phi\) is log-concave, for \(x,y\ge0\) we get
\[
        \frac{\Phi(b-x)}{\Phi(b)}
        \cdot
        \frac{\Phi(b-y)}{\Phi(b)}
        \ge
        \frac{\Phi(b-x-y)}{\Phi(b)}.
\]
Indeed, this follows by applying log-concavity of \(\Phi\) to the points
\(b\) and \(b-x-y\).  Using the standard identities for nonnegative
random variables,
\[
        \mathbb E W_b
        =
        \int_0^\infty \mathbb P(W_b\ge t)\,dt,
        \quad
        \mathbb E W_b^2
        =
        2\int_0^\infty t\,\mathbb P(W_b\ge t)\,dt,
\]
we obtain
\[
\begin{aligned}
        (\mathbb E W_b)^2
        &=
        \int_0^\infty\int_0^\infty
        \frac{\Phi(b-x)}{\Phi(b)}
        \frac{\Phi(b-y)}{\Phi(b)}
        \,dx\,dy                                      \\
        &\ge
        \int_0^\infty\int_0^\infty
        \frac{\Phi(b-x-y)}{\Phi(b)}
        \,dx\,dy                                      =
        \int_0^\infty
        t\frac{\Phi(b-t)}{\Phi(b)}\,dt
        =
        \frac12\mathbb E W_b^2 .
\end{aligned}
\]
Therefore
\[
        \operatorname{Var}(W_b)
        =
        \mathbb E W_b^2-(\mathbb E W_b)^2
        \le
        (\mathbb E W_b)^2.
\]
Since
\[
        \operatorname{Var}(W_b)=V(b),
        \quad
        \mathbb E W_b=b-\mathbb E[Z\mid Z\le b]=b+m(b),
\]
we obtain
\[
        V(b)\le (b+m(b))^2.
\]
Together with \(m(b)>0\), the formula for \(V'(b)\) gives \(V'(b)\ge 0\).

Moreover, if \(b\le0\), then the Mills ratio bound gives \(m(b)>-b\). Thus \(V(b)<1\) by noting $ V(b)-1= -bm(b)-m(b)^2 =m(b)(-b-m(b))<0$.   This completes the proof.
\end{proof}

The next lemma provides the sharp bound on the cumulant generating function that will replace the subgaussian estimate. 
A crucial observation is that the centered truncated Gaussian has zero mean and its variance is strictly smaller than 1, which is the key to our improvement.

\begin{lemma}[Cumulant generating function bound]\label{lem:CGF}
Let \(Z\sim\mathcal N(0,1)\) and \(b\in\mathbb{R}\). Define \(Y=Z-\mathbb{E}[Z\mid Z\le b]\). 
Then \(\mathbb{E}[Y\mid Z\le b]=0\), \(\operatorname{Var}(Y\mid Z\le b)=V(b)\), and for every \(u>0\),
\[
\log\mathbb{E}[e^{uY}\mid Z\le b]\le\frac{u^2}{2}V(b), \;\; \text{where} \;\; V(b):=\operatorname{Var}(Z\mid Z\le b).
\]
\end{lemma}

\begin{proof}
We start by computing the conditional moment generating function of \(Z\):
\[
\EE[e^{uZ}\mid Z\le b]=\frac{1}{\Phi(b)}\int_{-\infty}^b e^{uz}\phi(z)\,dz=e^{u^2/2}\frac{\Phi(b-u)}{\Phi(b)},
\]
by noting \(e^{uz}\phi(z)=e^{uz}\frac1{\sqrt{2\pi}}e^{-z^2/2}=e^{u^2/2}\phi(z-u)\).

Since \(\EE[Z\mid Z\le b]=-m(b)\), we have \(Y=Z+m(b)\) and therefore
\[
\EE[e^{uY}\mid Z\le b]=e^{um(b)}\EE[e^{uZ}\mid Z\le b]=e^{um(b)}e^{u^2/2}\frac{\Phi(b-u)}{\Phi(b)}.
\]

Define \(K(u)=\log\EE[e^{uY}\mid Z\le b]\). Then
\[
K(u)=\frac{u^2}{2}+um(b)+\log\Phi(b-u)-\log\Phi(b).
\]
Differentiating twice with respect to \(u\) gives
\[
K'(u)=u+m(b)-\frac{\phi(b-u)}{\Phi(b-u)}=u+m(b)-m(b-u),\quad
K''(u)=1+m'(b-u)=V(b-u),
\]
where the last equality uses the identity \(V(t)=1+m'(t)\) from the proof of Lemma~\ref{lem:truncated_variance_monotone}.

Observe that \(K(0)=K'(0)=0\). 
Integrating twice, since \(V\) is increasing by Lemma~\ref{lem:truncated_variance_monotone}, we obtain
\[
K(u)=\int_0^u K'(v)\,dv=\int_0^u\int_0^v K''(w)\,dw\,dv=\int_0^u\int_0^v V(b-w)\,dw\,dv\le V(b)\frac{u^2}{2}.
\]

Recalling that \(K(u)=\log\mathbb E[e^{uY}\mid Z\le b]\), we have proved the claimed inequality.
\end{proof}

\section{Improved independent set probability}
\label{sec:improved}

In this section, we incorporate the refined cumulant generating function bound into the HMS induction to obtain a sharper estimate for \(P_{\mathrm{red},r}^*\). 

Fix an exposure step \(s\) and set \(k=r-s\). Condition on the first \(s-1\) columns \(M[s-1]\) and on the diagonal entry \(y_s(s)\). 
For \(i=s+1,\ldots,r\), let \(X_i=y_i(s)\) and
\[
\mu_i=\EE[X_i\mid\overline{E}_{s,i},M[s-1],y_s(s)].
\]

Let $Z\sim\mathcal N(0,1)$ and set $b = -c_p + \varepsilon$ with $\varepsilon = O(D^{-1})$. 
To evaluate $\mu_i$, we consider the conditional expectation $\mathbb E[Z/\sqrt d \mid Z \le b]$. 
By Lemma~\ref{lem:trunc_exp},
\[
\mathbb E[Z/\sqrt d \mid Z \le b] = -\frac{\phi(b)}{\Phi(b)\sqrt d}.
\]
Since $\Phi(-c_p)=p$ and $\phi(-c_p)=a$, a Taylor expansion of the inverse Mills ratio around $-c_p$ gives
\[
\frac{\phi(b)}{\Phi(b)} = \frac{a}{p} + O(\varepsilon) = \frac{a}{p} + O(D^{-1}).
\]
Consequently,
\[
\mathbb E[Z/\sqrt d \mid Z \le b] = -\frac{a}{p\sqrt d} + O\left(\frac{1}{D\sqrt d}\right).
\]

In our setting, the conditioning event $\overline{E}_{s,i}$ together with the perfect sequence condition implies that $y_i(s)$ is a centered Gaussian with variance $1/d$, truncated from above at $-\frac{c_p}{\sqrt d} + O(\frac{1}{D\sqrt d})$. Therefore, for each $i>s$,
\[
\mu_i = \mathbb E\bigl[y_i(s) \mid \overline{E}_{s,i}, M[s-1], y_s(s)\bigr] = -\frac{a}{p\sqrt d} + O\left(\frac{1}{D\sqrt d}\right)<0.
\]

Define the quadratic form
\[
S=\sum_{s<i<j\le r}X_iX_j.
\]

We now sharpen the bound on \(\EE[e^{\lambda S}]\) from \cite{HMS}, where \(\lambda=-a\sqrt d/p\).
\begin{lemma}[Refined exponential moment]\label{lem:refined_S}
Under the above conditioning, with \(\lambda=-a\sqrt d/p\) and for sufficiently large \(D\),
\[
\log\EE[e^{\lambda S}]\le\lambda\EE[S]+(1-\gamma(p)+O(D^{-1}))\frac{\lambda^2(k-1)^2}{2d}\sum_{i=s+1}^r\mu_i^2+\frac{4|\lambda|k}{d}.
\]
\end{lemma}
\begin{proof}
Write \(X_i=\mu_i+\xi_i\), where \(\xi_i=X_i-\mu_i\) are centered. 
Expanding \(S\) gives
\[
S=\sum_{s<i<j\le r}\mu_i\mu_j+\sum_{i=s+1}^r A_i\xi_i+\sum_{s<i<j\le r}\xi_i\xi_j,
\]
where \(A_i=\sum\limits_{\substack{s<j\le r, j\ne i}}\mu_j<0\). The first term is deterministic, so it factors out of the expectation.

To handle the random terms, we apply H\"older's inequality. 
Choose exponents \(P_D=1+O(D^{-1})\) and \(Q_D=O(D)\) satisfying
\(P_D^{-1}+Q_D^{-1}=1\) and \(d\ge4Q_D|\lambda|k\). 
Then
\[
\EE[e^{\lambda S}]
\le
e^{\lambda\sum_{s<i<j\le r}\mu_i\mu_j}
\left(
\EE\left[
e^{P_D\lambda\sum_{i=s+1}^r A_i\xi_i}
\right]
\right)^{1/P_D}
\left(
\EE\left[
e^{Q_D\lambda\sum_{s<i<j\le r}\xi_i\xi_j}
\right]
\right)^{1/Q_D}.
\]

Under the conditioning in the lemma, the variables
\(\xi_{s+1},\ldots,\xi_r\) remain independent. For the linear term, conditional on $\overline{E}_{s,i}$ and the previous columns,
the scaled centered variable $\sqrt d\,\xi_i$ has the centered truncated
Gaussian distribution appearing in Lemma~\ref{lem:CGF}, with truncation
parameter \(b_i=-c_p+O(D^{-1})\). Hence $\mathbb{E}[\xi_i]=0$ and $\operatorname{Var}(\xi_i)=V(b_i)/d$.
Since \(\lambda<0\) and \(A_i<0\) for \(D\) sufficiently large, we have
\(P_D\lambda A_i/\sqrt d>0\). Applying Lemma~\ref{lem:CGF} with
\(u=P_D\lambda A_i/\sqrt d\) gives
\[
\log\mathbb{E}[e^{P_D\lambda A_i\xi_i}]
\le
\frac{P_D^2\lambda^2 A_i^2}{2d}\,V(b_i).
\]
Summing over \(i\) and using independence yields
\[
\frac1{P_D}
\log
\EE\left[
e^{P_D\lambda\sum_{i=s+1}^r A_i\xi_i}
\right]
\le
\frac{P_D\lambda^2}{2d}
\sum_{i=s+1}^r A_i^2V(b_i).
\]

Moreover, by Cauchy's inequality,
\[
A_i^2
=
\left(
\sum_{\substack{s<j\le r\\ j\ne i}}\mu_j
\right)^2
\le
(k-1)
\sum_{\substack{s<j\le r\\ j\ne i}}\mu_j^2 .
\]
Summing over \(i\) gives
\[
\sum_{i=s+1}^r A_i^2
\le
(k-1)^2
\sum_{i=s+1}^r\mu_i^2.
\]
Since
$
V(b_i)=V(-c_p)+O(D^{-1})
=
1-\gamma(p)+O(D^{-1})
$
and \(P_D=1+O(D^{-1})\), we obtain
\[
\frac1{P_D}
\log
\EE\left[
e^{P_D\lambda\sum_{i=s+1}^r A_i\xi_i}
\right]
\le
\bigl(1-\gamma(p)+O(D^{-1})\bigr)
\frac{\lambda^2(k-1)^2}{2d}
\sum_{i=s+1}^r\mu_i^2 .
\]

For the quadratic term, the \(\xi_i\) are independent and centered. 
Moreover, since each \(\xi_i\) is a truncated Gaussian, it is subgaussian with
variance proxy at most \(1/d\) (see \cite{Barreto}). Applying
Lemma~\ref{lem:quadratic} to the \(k=r-s\) variables
\(\xi_{s+1},\ldots,\xi_r\), with \(\lambda\) replaced by \(Q_D\lambda\), we obtain
\[
\begin{aligned}
\mathbb{E}\left[
e^{Q_D\lambda\sum_{s<i<j\le r}\xi_i\xi_j}
\right]
\le
\exp\left(
Q_D\lambda
\mathbb{E}\left[
\sum_{s<i<j\le r}\xi_i\xi_j
\right]
+
\frac{(Q_D\lambda)^2k^2}{d}
\sum_{i=s+1}^r(\mathbb{E}\xi_i)^2
+
\frac{4|Q_D\lambda|k}{d}
\right).
\end{aligned}
\]
Since \(\mathbb{E}[\xi_i]=0\) for every \(s<i\le r\), we have
$\mathbb{E}\left[
\sum_{s<i<j\le r}\xi_i\xi_j
\right]
=0$
and$
\sum_{i=s+1}^r(\mathbb{E}\xi_i)^2=0.$
Hence
\[
\mathbb{E}\left[
e^{Q_D\lambda\sum_{s<i<j\le r}\xi_i\xi_j}
\right]
\le
\exp\left(
\frac{4|Q_D\lambda|k}{d}
\right).
\]
This gives
$
\frac1{Q_D}
\log
\mathbb{E}\left[
e^{Q_D\lambda\sum_{s<i<j\le r}\xi_i\xi_j}
\right]
\le
\frac{4|\lambda|k}{d}.
$
Combining the three estimates proves the lemma.
\end{proof}

We now incorporate the refined cumulant generating function bound into the HMS induction. 
The following lemma isolates the reverse-induction step.  It is the point at which Lemma~\ref{lem:refined_S} replaces the subgaussian estimate used in the HMS argument.

To keep the induction formula readable, we introduce a few abbreviations.  For
\(0\le s\le r-1\), put
\[
\Sigma_s
:=
\sum_{s<i<j\le r}
\bigl\langle \pi_s(\by_i),\pi_s(\by_j)\bigr\rangle .
\]
We also write
\[
\beta_{\mathrm{quad}}(p):=\frac{\gamma(p)a^4}{8p^4}.
\]
For an integer \(t\ge1\), define the one-step error scale
\[
\mathcal E_t
:=
\frac{t^2}{D\sqrt d}
+
\frac{t^3}{Dd}
+
\frac{t}{\sqrt d}
+
t\delta .
\]
Finally, for \(0\le s\le r-1\), set
\[
\mathcal R_{s,r}
:=
\sum_{m=s+1}^{r-1}
\left[
\bigl(1-\gamma(p)\bigr)\frac{a^4}{2p^4d}(r-m)(r-m-1)^2
+
O\bigl(\mathcal E_{r-m}\bigr)
\right],
\]
where an empty sum is understood to be zero.

\begin{lemma}[Refined Bartlett induction step]\label{lem:refined_induction_step}
Fix the first \(s\) columns \(M[s]\) and let \(k = r-s\). Under the perfect sequence event \(B_r\), the conditional probability that vertices \(\{s+1,\dots,r\}\) form a red clique (independent set) satisfies
\[
\mathbb{P}[I_s \wedge B_r \mid M[s]]
\le
p^{\binom{k}{2}}
\exp\left(
-\frac{a\sqrt d}{p}\Sigma_s
-
\frac{a^3}{p^3\sqrt d}\binom{k}{3}
+
\mathcal R_{s,r}
\right).
\]
\end{lemma}

\begin{proof}[Proof of Lemma~\ref{lem:refined_induction_step}]
We proceed by reverse induction on \(s\). For the base case \(s = r-1\) we have \(k = 1\), and a single vertex automatically forms an independent set.  Since \(\Sigma_{r-1}=0\) and \(\mathcal R_{r-1,r}=0\), the right-hand side is \(1\), so the bound holds trivially. For the inductive step, assume the bound for \(s\), corresponding to size \(k\), and prove it for \(s-1\), corresponding to size \(k+1\).

By the law of total probability, conditioning on the \(s\)-th column \(M_s\) and the event that vertex \(s\) has no blue edges to later vertices, we have
\begin{equation}\label{eq:con-total}
\begin{aligned}
\mathbb{P}(I_{s-1} \wedge B_r \mid M[s-1])
&=
\mathbb{E}_{M_s}\bigg[
\mathbb{P}(I_s \wedge B_r \mid M[s])
\Bigm|
\bigwedge_{i=s+1}^r \overline{E}_{s,i},M[s-1]
\bigg] \\
&\quad \cdot
\mathbb{P}\bigg(
\bigwedge_{i=s+1}^r \overline{E}_{s,i}
\Bigm| M[s-1]
\bigg).
\end{aligned}
\end{equation}

In the inductive step for a given \(s\), the column \(M_s\) is already exposed. Moreover, under the conditioning \(\bigwedge_{i=s+1}^r \overline{E}_{s,i}\), the coordinates \(y_{s+1}(s),\ldots,y_r(s)\) are independent and follow the upper-truncated Gaussian distribution as in Lemma~\ref{lem:refined_S}.

From Claim~\ref{claim:connection}, the second factor in \eqref{eq:con-total} is bounded by
\begin{equation}\label{eq:iso}
\begin{aligned}
\mathbb{P}\left(
\bigwedge_{i=s+1}^r \overline{E}_{s,i}
\Bigm| M[s-1]
\right) \le
p^{k}
\exp\left(
-\frac{a\sqrt d}{p}
\sum_{i=s+1}^r
\bigl\langle \pi_{s-1}(\by_s),\pi_{s-1}(\by_i)\bigr\rangle
+
O(k\delta)
\right).
\end{aligned}
\end{equation}
The factor \(O(k\delta)\) will be included in the one-step error \(O(\mathcal E_k)\).

To evaluate the first factor in \eqref{eq:con-total}, decompose the inner product sum in the induction hypothesis using orthogonal projection:
\begin{equation}\label{con-f}
\Sigma_s
=
\sum_{s<i<j\le r}
\bigl\langle \pi_{s-1}(\by_i),\pi_{s-1}(\by_j)\bigr\rangle
+
\underbrace{
\sum_{s<i<j\le r} y_i(s)y_j(s)
}_{:=S}.
\end{equation}
Substituting this into the induction hypothesis and isolating the \(M_s\)-dependent part gives
\begin{align}\label{first-f}
\mathbb{E}_{M_s}[\cdot]
\le
p^{\binom{k}{2}}
\exp\Bigg(
-\frac{a\sqrt d}{p}
\sum_{s<i<j\le r}
\bigl\langle \pi_{s-1}(\by_i),\pi_{s-1}(\by_j)\bigr\rangle
-
\frac{a^3}{p^3\sqrt d}\binom{k}{3}
+
\mathcal R_{s,r}
\Bigg)
\cdot
\mathbb{E}_{M_s}[e^{\lambda S}],
\end{align}
where \(\lambda = -a\sqrt d/p\).

Now apply Lemma~\ref{lem:refined_S}.  For sufficiently large \(D\),
\begin{equation}\label{eq:mom}
\begin{aligned}
\log \mathbb{E}_{M_s}[e^{\lambda S}]
&\le
\lambda\mathbb{E}[S]
+
\bigl(1-\gamma(p)+O(D^{-1})\bigr)
\frac{\lambda^2(k-1)^2}{2d}
\sum_{i=s+1}^r\mu_i^2
+
\frac{4|\lambda|k}{d},
\end{aligned}
\end{equation}
where \(\mu_i = \mathbb{E}[y_i(s) \mid \overline{E}_{s,i}, M[s-1], y_s(s)]\).
By Lemma~\ref{lem:trunc_exp} and the perfect sequence estimates,
\[
\mu_i = -\frac{a}{p\sqrt d}+O\left(\frac{1}{D\sqrt d}\right).
\]
Substituting \(\lambda = -a\sqrt d/p\) and this expansion into \eqref{eq:mom} gives
\begin{equation}\label{eq:lead}
\lambda\mathbb{E}[S]
=
\lambda\sum_{s<i<j\le r}\mu_i\mu_j
=
-\frac{a^3}{p^3\sqrt d}\binom{k}{2}
+
O\left(\frac{k^2}{D\sqrt d}\right),
\end{equation}
and
\begin{equation}\label{eq:var}
\begin{aligned}
\bigl(1-\gamma(p)+O(D^{-1})\bigr)
\frac{\lambda^2(k-1)^2}{2d}
\sum_{i=s+1}^r\mu_i^2 =
\bigl(1-\gamma(p)\bigr)\frac{a^4}{2p^4d}k(k-1)^2
+
O\left(\frac{k^3}{Dd}\right).
\end{aligned}
\end{equation}
Plugging \eqref{eq:lead} and \eqref{eq:var} into \eqref{eq:mom} yields
\begin{equation}\label{eq:mom-2}
\mathbb{E}_{M_s}[e^{\lambda S}]
\le
\exp\left(
-\frac{a^3}{p^3\sqrt d}\binom{k}{2}
+
\bigl(1-\gamma(p)\bigr)\frac{a^4}{2p^4d}k(k-1)^2
+
O\left(
\frac{k^2}{D\sqrt d}
+
\frac{k^3}{Dd}
+
\frac{k}{\sqrt d}
\right)
\right).
\end{equation}

Substitute \eqref{eq:iso} and \eqref{eq:mom-2} into \eqref{eq:con-total}, and also incorporate the exponential factors from \eqref{first-f}.  Since
\[
p^k p^{\binom{k}{2}}=p^{\binom{k+1}{2}},
\qquad
\binom{k}{3}+\binom{k}{2}=\binom{k+1}{3},
\]
and since the two geometric sums combine as
\[
\sum_{s<i<j\le r}
\bigl\langle \pi_{s-1}(\by_i),\pi_{s-1}(\by_j)\bigr\rangle
+
\sum_{i=s+1}^r
\bigl\langle \pi_{s-1}(\by_s),\pi_{s-1}(\by_i)\bigr\rangle
=
\Sigma_{s-1},
\]
we obtain
\[
\begin{aligned}
\mathbb{P}(I_{s-1} \wedge B_r \mid M[s-1])
&\le
p^{\binom{k+1}{2}}
\exp\left(
-\frac{a\sqrt d}{p}\Sigma_{s-1}
-
\frac{a^3}{p^3\sqrt d}\binom{k+1}{3}
\right.\\
&\quad\qquad\left.
+
\mathcal R_{s,r}
+
\bigl(1-\gamma(p)\bigr)\frac{a^4}{2p^4d}k(k-1)^2
+
O\bigl(\mathcal E_k\bigr)
\right).
\end{aligned}
\]
Here \(k=r-s\), so by the definition of \(\mathcal R_{s,r}\),
\[
\mathcal R_{s-1,r}
=
\mathcal R_{s,r}
+
\bigl(1-\gamma(p)\bigr)\frac{a^4}{2p^4d}(r-s)(r-s-1)^2
+
O\bigl(\mathcal E_{r-s}\bigr).
\]
Thus the last display is exactly the desired bound for \(s-1\), completing the induction step.
\end{proof}

The following proposition is the core technical improvement of this paper, which introduces an extra negative term ``\(-\beta_{\mathrm{quad}}(p) r^4/d\)'' in the exponent.

\begin{proposition}[Improved independent set probability]\label{prop:improved_ind}
For any \(C>1\), let \(d = D^2\ell^2\) with \(D\) sufficiently large. Then for all \(1 \le r \le C\ell\), the probability \(P_{\mathrm{red}, r}^*\) that the vectors \(\bx_1,\dots,\bx_r\) form a red clique (independent set) and are perfect satisfies
\[
P_{\mathrm{red}, r}^*
\le
p^{\binom{r}{2}}
\exp\left(
-\frac{a^3}{p^3\sqrt d} \binom{r}{3}
+
K\left(
\frac{r^3}{D\sqrt d}
+
\frac{r^2}{\sqrt d}
+
r^2\delta
\right)
-
\beta_{\mathrm{quad}}(p)\frac{r^4}{d}
+
O(D^{-1})\frac{r^4}{d}
\right),
\]
where \(a = \phi(c_p)\), \(K = K(p) > 0\) depends only on \(p\), \(\delta=\alpha d^{-1/4}\) is the perfection parameter from Definition~\ref{def:perfect}, and \(\beta_{\mathrm{quad}}(p)=\gamma(p)a^4/(8p^4)>0\). Here \(\gamma(p) = 1 - \operatorname{Var}(Z \mid Z \le -c_p) > 0\). Moreover, the same bound holds for the unconditional red clique probability \(P_{\mathrm{red},r}\), since the probability of failing to be perfect is exponentially small and can be absorbed into the error term.
\end{proposition}

\begin{proof}
We apply Lemma~\ref{lem:refined_induction_step} with \(s=0\). At \(s=0\), no columns are exposed and \(\pi_0\) is the zero map, so \(\Sigma_0=0\). Hence
\[
\mathbb{P}(I_0 \wedge B_r)
\le
p^{\binom{r}{2}}
\exp\left(
-\frac{a^3}{p^3\sqrt d}\binom{r}{3}
+
\mathcal R_{0,r}
\right).
\]
By the definition of \(\mathcal R_{0,r}\),
\[
\mathcal R_{0,r}
=
\sum_{m=1}^{r-1}
\bigl(1-\gamma(p)\bigr)\frac{a^4}{2p^4d}(r-m)(r-m-1)^2
+
\sum_{m=1}^{r-1}O\bigl(\mathcal E_{r-m}\bigr).
\]
Putting \(t=r-m\), we have
\[
\sum_{m=1}^{r-1}(r-m)(r-m-1)^2
=
\sum_{t=1}^{r-1}t(t-1)^2
=
\frac{r^4}{4}+O(r^3).
\]
Therefore
\[
\begin{aligned}
\sum_{m=1}^{r-1}
\bigl(1-\gamma(p)\bigr)\frac{a^4}{2p^4d}(r-m)(r-m-1)^2
=
\sum_{m=1}^{r-1}
\frac{a^4}{2p^4d}(r-m)(r-m-1)^2 -
\beta_{\mathrm{quad}}(p)\frac{r^4}{d}
+
O\left(\frac{r^3}{d}\right).
\end{aligned}
\]
The accumulated error terms satisfy
\[
\sum_{m=1}^{r-1}O\bigl(\mathcal E_{r-m}\bigr)
=
O\left(\frac{r^3}{D\sqrt d}\right)
+
O\left(\frac{r^4}{Dd}\right)
+
O\left(\frac{r^2}{\sqrt d}\right)
+
O(r^2\delta).
\]
Since \(d=D^2\ell^2\) and \(r\le C\ell\), the term \(O(r^2/\sqrt d)\) is lower order in the final application \(r=\ell\), and the additional \(O(r^2\delta)\) term is also lower order after normalization by \(\ell^2\).  We nevertheless keep it explicitly in the probability estimate.

The first term
\[
\sum_{m=1}^{r-1}
\frac{a^4}{2p^4d}(r-m)(r-m-1)^2
\]
together with the leading term \(-a^3\binom r3/(p^3\sqrt d)\) reproduces the main HMS exponential factor, up to the usual HMS error
\[
K\left(
\frac{r^3}{D\sqrt d}
+
\frac{r^2}{\sqrt d}+r^2\delta
\right).
\]
The new negative contribution is exactly \(-\beta_{\mathrm{quad}}(p)r^4/d\). Thus
\[
\mathbb{P}(I_0 \wedge B_r)
\le
p^{\binom{r}{2}}
\exp\left(
-\frac{a^3}{p^3\sqrt d}\binom{r}{3}
+
K\left(
\frac{r^3}{D\sqrt d}
+
\frac{r^2}{\sqrt d}
+
r^2\delta
\right)
-
\beta_{\mathrm{quad}}(p)\frac{r^4}{d}
+
O(D^{-1})\frac{r^4}{d}
\right).
\]

Finally, as in the extraction argument of~\cite{HMS}, the same bound holds for the unconditional probability \(P_{\mathrm{red},r}\), with the resulting constant absorbed into \(O(D^{-1})r^4/d\). This is done via an extraction procedure: given the sequence 
$\bx_1,\dots,\bx_r$ in order, we greedily select a subsequence 
$\by_1,\dots,\by_t$ as follows. For each $i$, if $\bx_i$ satisfies 
$\|\bx_i\|_2\in(1-\delta,1+\delta)$ and 
$\|\pi_{\operatorname{span}\{\by_1,\dots,\by_j\}}(\bx_i)\|_2 \le \alpha\sqrt{\ell}/\sqrt{d}$, 
we add it to the subsequence; otherwise we discard it. The resulting 
subsequence is perfect by construction. 

Now suppose the original sequence forms a red clique.  If a vertex \(x_i\) is discarded, then by definition either its norm deviates from \(1\) by more than \(\delta\), or its projection onto the span of previously kept vectors is too large.  By Lemma~\ref{lem:norm} and Lemma~\ref{lem:proj}, each such event occurs with probability at most \((p/10)^{10C\ell}\), even conditionally on the previous vectors.  Hence, the probability that a given set \(I\) of indices is kept and the rest discarded is at most
$
        P^*_{\mathrm{red},|I|}
        \left(p/{10}\right)^{10C\ell(r-|I|)}.
$
If \(|I|=r-u\), then replacing \((r-u)^4\) by \(r^4\) costs at most
$
        \beta_{\rm quad}(p)\frac{r^4-(r-u)^4}{d}
        \le
        \frac{4\beta_{\rm quad}(p)C^3}{D^2}\ell u,
$
which is absorbed by \((p/10)^{10C\ell u}\) when \(D\) is sufficiently large.  Summing over all possible \(I\), the contributions from non-perfect sequences are dominated by the case \(I=[r]\), contributing only a constant that can be absorbed into \(O(D^{-1})r^4/d\).  Consequently, the same upper bound holds for \(P_{\mathrm{red},r}\).  This completes the proof of Proposition~\ref{prop:improved_ind}.
\end{proof}

\section{Improved Ramsey lower bound}
\label{sec:proof}

With Proposition~\ref{prop:improved_ind} established, we now complete the proof of
Theorem~\ref{thm:main}.

\begin{proof}[Proof of Theorem~\ref{thm:main}]
Set \(d=D^2\ell^2\), where \(D\) is a sufficiently large constant.  Recall
that \(p_C\in(0,1/2)\) is defined by
\[
C=\frac{\log p_C}{\log(1-p_C)}.
\]
Equivalently, \(-\frac12\log p_C=-\frac C2\log(1-p_C)\).  All estimates below
are uniform for \(p\) in a fixed neighborhood of \(p_C\).

Let \(n=\exp(\rho\ell)\).  By the first moment method,
\[
\mathbb E[\#\text{ red }K_\ell]\le \binom n\ell P_{\mathrm{red},\ell},
\quad
\mathbb E[\#\text{ blue }K_{C\ell}]\le \binom n{C\ell}P_{\mathrm{blue},C\ell}.
\]

By Proposition~\ref{prop:perfect_bound},
with \(r=\ell\) and \(d=D^2\ell^2\), the HMS red bound can be written as
\[
P_{\mathrm{red},\ell}^{\mathrm{HMS}}
\le
p^{\binom{\ell}{2}}
\exp\left(
-\frac{a^3}{p^3\sqrt d}\binom{\ell}{3}
+
K\frac{\ell^3}{D\sqrt d}
\right)
:=
\exp\left(
-\rho_{\mathrm{red}}^{\mathrm{HMS}}(p,D;\ell)\ell^2
\right).
\]
Thus
\[
\begin{aligned}
\rho_{\mathrm{red}}^{\mathrm{HMS}}(p,D;\ell)
&=
-\frac1{\ell^2}
\log\left[
p^{\binom{\ell}{2}}
\exp\left(
-\frac{a^3}{p^3\sqrt d}\binom{\ell}{3}
+
K\frac{\ell^3}{D\sqrt d}
\right)
\right]  \\
&=
-\frac{\binom{\ell}{2}}{\ell^2}\log p
+
\frac{a^3}{p^3D}\frac{\binom{\ell}{3}}{\ell^3}
-
\frac{K}{D^2}  \\
&=
-\frac12\log p
+
\frac{\kappa(p)}{D}
-
\frac{K}{D^2}
+
o_\ell(1),
\end{aligned}
\]
where \(\kappa(p)=a^3/(6p^3)\).

By Proposition~\ref{prop:improved_ind}, the refined red estimate gives
\begin{align*}
P_{\mathrm{red},\ell}^{\mathrm{new}}
&\le
p^{\binom{\ell}{2}}
\exp\left(
-\frac{a^3}{p^3\sqrt d} \binom{\ell}{3}
+
K\left(
\frac{\ell^3}{D\sqrt d}
+
\frac{\ell^2}{\sqrt d}
+
\ell^2\delta
\right)
-
\beta_{\mathrm{quad}}(p)\frac{\ell^4}{d}
+
O(D^{-1})\frac{\ell^4}{d}
\right)
\\&:=
\exp\left(
-\rho_{\mathrm{red}}^{\mathrm{new}}(p,D;\ell)\ell^2
\right).
\end{align*}
The term \(K\ell^2\delta\) contributes \(O(\delta)=o_\ell(D^{-2})\) to the normalized exponent, since \(\delta=\alpha d^{-1/4}=O(D^{-1/2}\ell^{-1/2})\) and \(D\) is fixed while \(\ell\to\infty\).  More importantly, the comparison with HMS is made before the remaining second-order terms are absorbed into the generic constant \(K\).  In the HMS induction the corresponding linear-exponential-moment contribution is
\[
\frac{a^4}{2p^4d}k(k-1)^2,
\]
whereas the refined cumulant generating function estimate replaces it by
\[
\bigl(1-\gamma(p)+O(D^{-1})\bigr)
\frac{a^4}{2p^4d}k(k-1)^2.
\]
All other terms agree with the HMS expansion up to
\(O(D^{-1})r^4/d+o_\ell(r^4/d)\).  Thus the refined estimate improves the red exponent, relative to the HMS estimate, by
\[
\frac{\gamma(p)a^4}{8p^4}\frac{r^4}{d}
+
O(D^{-1})\frac{r^4}{d}
+
o_\ell\left(\frac{r^4}{d}\right).
\]
Consequently,
\[
\begin{aligned}
\rho_{\mathrm{red}}^{\mathrm{new}}(p,D;\ell)
&=
\rho_{\mathrm{red}}^{\mathrm{HMS}}(p,D;\ell)
+
\beta_{\mathrm{quad}}(p)\frac{\ell^2}{d}
+
O(D^{-1})\frac{\ell^2}{d}
+
O(\delta)
+
O\left(\frac{1}{D\ell}\right)
+
o\left(\frac{\ell^2}{d}\right)
\\
&=
\rho_{\mathrm{red}}^{\mathrm{HMS}}(p,D;\ell)
+
\frac{\beta_{\mathrm{quad}}(p)}{D^2}
+
O(D^{-3})
+
o_\ell(D^{-2}).
\end{aligned}
\]
Equivalently, after letting \(\ell\to\infty\) with \(D\) fixed,
\[
\rho_{\mathrm{red}}^{\mathrm{new}}(p,D)
=
\rho_{\mathrm{red}}^{\mathrm{HMS}}(p,D)
+
\frac{\beta_{\mathrm{quad}}(p)}{D^2}
+
O(D^{-3}).
\]

We keep the blue-side estimate from the HMS analysis unchanged.  By Proposition~\ref{prop:perfect_bound} with \(r=C\ell\), and writing
\(\Lambda(p)=a^3/(6(1-p)^3)\), the normalized blue exponent is
\[
\rho_{\mathrm{blue}}(p,D)
= -\frac C2\log(1-p)-\frac{C^2\Lambda(p)}{D}+O(D^{-2}).
\]

Define
\[
\rho^{\mathrm{new}}(D)
:=
\max_p
\min\{\rho_{\mathrm{red}}^{\mathrm{new}}(p,D),\rho_{\mathrm{blue}}(p,D)\}.
\]

We now compare the two optimized exponents.  The leading red and blue curves
are
\[
r_0(p)=-\frac12\log p,
\quad
b_0(p)=-\frac C2\log(1-p),
\]
and they cross at \(p=p_C\). The crossing is transverse since
\[
r_0'(p_C)=-\frac1{2p_C}<0 
\quad \text{and} \quad
b_0'(p_C)=\frac{C}{2(1-p_C)}>0
\]
have opposite signs and are therefore unequal.

The HMS and refined optimizers lie at \(p_C+O(D^{-1})\) (see \cite[Lemma~5.1]{HMS} for the HMS case; the refined case follows similarly because the correction is of higher order).  
The refined estimate adds $\beta_{\mathrm{quad}}(p)/D^2$ to the red exponent while leaving the blue exponent unchanged. At the crossing $p=p_C$, the red and blue leading-order exponents satisfy $r_0(p_C)=b_0(p_C)$. A vertical shift of the red curve by $\delta = \beta_{\mathrm{quad}}(p_C)/D^2 + O(D^{-3})$ increases the optimum by $\lambda_C \delta + O(\delta^2)$,\footnote{Consider the leading-order exponents $r_0(p)$ and $b_0(p)$. At the crossing $p=p_C$ we have $r_0(p_C)=b_0(p_C)$. Adding a vertical shift $\delta$ to $r_0$ and setting $r_0(p_C+h)+\delta = b_0(p_C+h)$ gives $r_0'(p_C)h + \delta = b_0'(p_C)h + O(h^2)$. Solving yields $h = \delta/(b_0'(p_C)-r_0'(p_C)) + O(\delta^2)$, so the optimum increases by $\lambda_C\delta + O(\delta^2)$ with $\lambda_C = b_0'(p_C)/(b_0'(p_C)-r_0'(p_C))$.} 
where $$\lambda_C = b_0'(p_C)/(b_0'(p_C)-r_0'(p_C))= \frac{Cp_C}{Cp_C+(1-p_C)}> 0$$ is the sensitivity factor (See Figure~\ref{fig:optimization}). Since the HMS optimum is $\rho^{\mathrm{HMS}}(D) = r_0(p_C) + O(D^{-1})$, we obtain
\[
\rho^{\mathrm{new}}(D) = \rho^{\mathrm{HMS}}(D) + \lambda_C\frac{\beta_{\mathrm{quad}}(p_C)}{D^2} + O(D^{-3}),
\]
where $\rho^{\mathrm{HMS}}(D)=\max_p \min\{\rho_{\mathrm{red}}^{\mathrm{HMS}}(p,D),\rho_{\mathrm{blue}}(p,D)\}$.

\medskip
\begin{figure}[t]
\centering
\begin{tikzpicture}[scale=1, >=stealth]
  \draw[->] (-0.5,0) -- (6.0,0) node[below] {$p$};
  \draw[->] (0,-1.5) -- (0,3.2) node[left] {$\rho$};

  \def\pC{2.5}
  \def\redY{0.8}
  \def\blueSlope{0.65}
  \def\redSlope{-0.85}

  \draw[blue, thick, domain=0.5:5.2] plot (\x, {\blueSlope*(\x-\pC) + \redY})
        node[below right] {$\rho_{\mathrm{blue}}(p)$};

  \draw[red, thick, domain=0.5:5.2] plot (\x, {\redSlope*(\x-\pC) + \redY})
        node[below] {$\rho_{\mathrm{red}}(p)$};

  \draw[red, thick, dashed, domain=0.5:5.2] plot (\x, {\redSlope*(\x-\pC) + \redY + 0.9})
        node[below right] {$\rho_{\mathrm{red}}^{\mathrm{new}}(p)$};

  \fill[black] (\pC, \redY) circle (2pt) 
        node[left, xshift=-1.5pt] {$(p_C^{\mathrm{HMS}},\;\rho^{\mathrm{HMS}})$};

  \def\newShift{0.6}
  \def\newP{\pC+\newShift}
  \def\newRho{\redY + 0.9 + \redSlope*\newShift}
  \fill[black] (\newP, \newRho) circle (2pt) 
        node[right, xshift=2pt] {$(p_C^{\mathrm{HMS}}+h,\;\rho^{\mathrm{HMS}}+\Delta)$};

\end{tikzpicture}
\caption{Comparison of HMS and refined red exponents. The refined red curve $\rho_{\mathrm{red}}^{\mathrm{new}}(p)$ (dashed) intersects the blue curve at $(p_C^{\mathrm{HMS}}+h,\rho^{\mathrm{HMS}}+\Delta)$, where \(p_C^{\mathrm{HMS}}\) denotes the optimal $p$ in the HMS construction, $h = O(D^{-1})$  is a small perturbation and $\Delta = \lambda_C \beta_{\mathrm{quad}}(p_C)D^{-2} + O(D^{-3})$.}
\label{fig:optimization}
\end{figure}
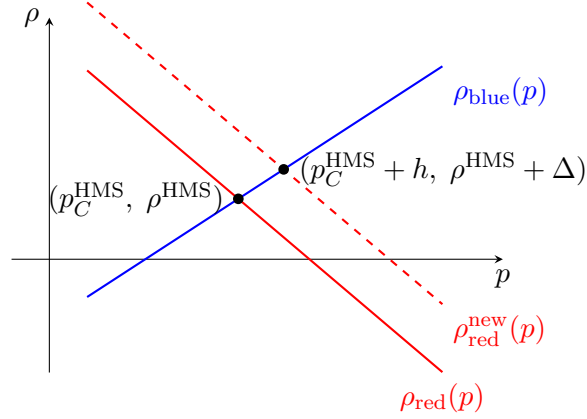

\medskip
Now choose \(n=\exp((\rho^{\mathrm{new}}(D)-o(1))\ell)\). Then from the definition of $\rho^{\mathrm{new}}(D)$,
\[
P_{\mathrm{red},\ell}^{\mathrm{new}}\le\exp\left(-\rho^{\mathrm{new}}(D)\cdot\ell^2\right), \quad P_{\mathrm{blue},C\ell}\le\exp\left(-C\rho^{\mathrm{new}}(D)\cdot\ell^2\right).
\]
Thus, $\mathbb E[\#\text{ red }K_\ell]\le \binom n\ell P_{\mathrm{red},\ell}\to 0,$ and
$\mathbb E[\#\text{ blue }K_{C\ell}]\le \binom n{C\ell}P_{\mathrm{blue},C\ell}\to 0$ by a direct calculation.
 Therefore, there is a red-blue coloring of \(K_n\) with no red \(K_\ell\) and no blue \(K_{C\ell}\), implying that
\[
R(\ell,C\ell)
\ge n=
\exp((\rho^{\mathrm{new}}(D)-o(1))\ell).
\]

We now write this in the form stated in Theorem~\ref{thm:main}.  For this
fixed choice of \(D\), write
\[
p_C^{-1/2}+\varepsilon_{\mathrm{HMS}}
=
\exp(\rho^{\mathrm{HMS}}(D)),
\]
where $\varepsilon_{\mathrm{HMS}}>0$ is the constant obtained in \cite{HMS}.
Then
\[
\begin{aligned}
\exp(\rho^{\mathrm{new}}(D))
&=
\left(p_C^{-1/2}+\varepsilon_{\mathrm{HMS}}\right)
\exp\left(
\lambda_C\frac{\beta_{\mathrm{quad}}(p_C)}{D^2}
+
O(D^{-3})
\right) \\
&=
\left(p_C^{-1/2}+\varepsilon_{\mathrm{HMS}}\right)
\left(
1+
\lambda_C\frac{\beta_{\mathrm{quad}}(p_C)}{D^2}
+
O(D^{-3})
\right).
\end{aligned}
\]
Thus we can rephrase the lower bound as
\[
R(\ell,C\ell)\ge \left(p_C^{-1/2}+\eta\right)^\ell,
\]
where
\[
\eta
=
\varepsilon_{\mathrm{HMS}}
+
\left(p_C^{-1/2}+\varepsilon_{\mathrm{HMS}}\right)
\lambda_C
\frac{\beta_{\mathrm{quad}}(p_C)}{D^2}
+
O(D^{-3}p_C^{-1/2}).
\]
Recall that $\lambda_C=\frac{Cp_C}{Cp_C+(1-p_C)}$
and $\beta_{\mathrm{quad}}(p_C)
=
\frac{\gamma(p_C)a^4}{8p_C^4},
$ with $\gamma(p_C)
=
1-\operatorname{Var}(Z\mid Z\le -c_{p_C})>0$. Therefore, for \(D\) sufficiently large, the error term
\(O(D^{-3}p_C^{-1/2})\) is dominated by half of the positive \(D^{-2}\)-term. Hence
\[
\eta
\ge
\varepsilon_{\mathrm{HMS}}
+
\frac12
\left(p_C^{-1/2}+\varepsilon_{\mathrm{HMS}}\right)
\lambda_C
\frac{\beta_{\mathrm{quad}}(p_C)}{D^2}
>
\varepsilon_{\mathrm{HMS}}.
\]

\smallskip
\noindent\textbf{$\bullet$ Asymptotic form as \(C\to\infty\).}
We use the parameters from Appendix~B of \cite{HMS}:
$
        p=p_C+\frac1C$, and
               $D=\frac{4aC}{1-p},
$
where \(a=\phi(c_p)\).  It is known \cite{HMS} that
\[
        \exp(\rho^{\mathrm{HMS}}(D))
        =
        e^{1/24+o(1)}p_C^{-1/2}.
\]
Combining this with the estimate obtained above, we have
\[
        \rho^{\mathrm{new}}(D)
        =
        \rho^{\mathrm{HMS}}(D)
        +
        (1+o(1))
        \lambda_C
        \frac{\gamma(p_C)\phi(c_{p_C})^4}{8p_C^4D^2}
        +
        O(D^{-3}).
\]
Since
$\lambda_C
        =
        \frac{Cp_C}{Cp_C+(1-p_C)}
        =
        1+o(1)$ when $C\to\infty$.

As \(C\to\infty\),
$
        p_C\sim\frac{\log C}{C},
        \quad
        c_{p_C}^2\sim2\log\frac1{p_C}.
$
Since \(p=p_C+1/C\) and \(1/C=o(p_C)\), we have
$
        a=\phi(c_p)\sim\phi(c_{p_C})\sim c_{p_C}p_C.
$
Therefore
$        D
        =
        \frac{4aC}{1-p}
        \sim
        4c_{p_C}Cp_C
        \sim
        4c_{p_C}\log C
        \sim
        2c_{p_C}^3.
$
Also, \(\gamma(p_C)\to1\), and hence
$
        \frac{\gamma(p_C)\phi(c_{p_C})^4}{8p_C^4}
        \sim
        \frac{c_{p_C}^4}{8}.
$
It follows that
\[
        \frac{\gamma(p_C)\phi(c_{p_C})^4}{8p_C^4D^2}
        \sim
        \frac{c_{p_C}^4/8}{4c_{p_C}^6}
        =
        \frac1{32c_{p_C}^2}
        \sim
        \frac1{64\log(1/p_C)}
        \sim
        \frac1{64\log C}.
\]
Since \(D=\Theta((\log C)^{3/2})\), we also have
$
        D^{-3}=o(\frac1{\log C}).
$
Consequently,
\[
        \rho^{\mathrm{new}}(D)
        =
        -\frac12\log p_C
        +
        \frac1{24}
        +
        \frac1{64\log C}
        +
        o\left(\frac1{\log C}\right).
\]
Exponentiating gives
\[        \exp(\rho^{\mathrm{new}}(D))
        =e^{1/24}p_C^{-1/2}\left(1+\frac1{64\log C}
        +
        o\left(\frac1{\log C}\right)
        \right).
\]
Equivalently, 
\[
        \eta
        =
        (e^{1/24}-1)p_C^{-1/2}
        +
        \frac{e^{1/24}}{64}
        \frac{p_C^{-1/2}}{\log C}
        +
        o\left(\frac{p_C^{-1/2}}{\log C}\right).
\]
This completes the proof of Theorem \ref{thm:main}.
\end{proof}

\section{Concluding remarks}

A similar refinement can also be applied to the blue cliques, using a lower-truncated version of Lemma~\ref{lem:CGF} (see Lemma~\ref{lem:local_lower_CGF} below). Incorporating both improvements would yield a slightly larger explicit constant, but does not affect the qualitative statement that $\eta > \epsilon_{\mathrm{HMS}}$. Here $D$ is the scaling parameter used above, with $d=D^2\ell^2$ being the dimension, and we take $D\to\infty$. 

\begin{lemma}[Lower-truncated estimate]\label{lem:local_lower_CGF}
Let $Z\sim\mathcal N(0,1)$, $b\in\mathbb R$, and let
$
Y_+^{(b)}:=Z-\mathbb E[Z\mid Z\ge b].
$
For every fixed $M>0$, uniformly for $|u|\le MD^{-1}$,
\[
\log\mathbb E\left[e^{uY_+^{(b)}}\mid Z\ge b\right]
\le
\frac{u^2}{2}
\left(
\operatorname{Var}(Z\mid Z\ge b)+O_{b,M}(D^{-1})
\right).
\]
\end{lemma}
\begin{proof}[Sketch of proof]
Let $K_+(u):=\log\mathbb E[e^{uY_+^{(b)}}\mid Z\ge b]$.  A direct computation of the moment generating function for a lower-truncated normal gives $K_+''(u)=\operatorname{Var}(Z\mid Z\ge b-u)$.  For $|u|\le MD^{-1}$, the cutoff $b-u$ lies in a fixed compact neighbourhood of $b$.  Since the truncated variance is smooth as a function of the cutoff, we have $K_+''(u)=\operatorname{Var}(Z\mid Z\ge b)+O_{b,M}(D^{-1})$.  The result follows by integrating twice and using $K_+(0)=K_+'(0)=0$.
\end{proof}

\noindent
With this estimate, an induction completely analogous to the red case yields an extra negative term $-\beta_{\mathrm{quad}}'(p)\,r^4/d$ in the blue clique probability, where $\beta_{\mathrm{quad}}'(p)=\frac{\gamma'(p)a^4}{8(1-p)^4}$ and $\gamma'(p)=1-\operatorname{Var}(Z\mid Z\ge -c_p)>0$.  Including both improvements would increase the final constant $\eta$ further, but since our main theorem already establishes $\eta>\epsilon_{\mathrm{HMS}}$ (a strict improvement over \cite{HMS}) and the asymptotic expression as $C\to\infty$ remains unchanged to leading order, we omit the explicit computation for the sake of brevity.

\bibliographystyle{plain}

\end{document}